\documentclass[12pt]{article}
\expandafter\let\csname equation*\endcsname\relax

\expandafter\let\csname endequation*\endcsname\relax

\usepackage{amsmath}
\usepackage{amsthm}
\usepackage{xcolor}
\usepackage{amssymb}
\usepackage{multirow}
\usepackage{mathtools}
\usepackage{thmtools}
\usepackage{hyperref}
\usepackage{cleveref}

\def\bn#1{\mathbf{#1}}
\newcommand{\Set}[2]{%
	\{\, #1 \mid #2 \, \}%
}
\newcommand{\df}{\stackrel{def}{=}}
\def \r{\mathbf{r}}

\newtheorem{theorem}{Theorem}

\newtheorem{claim}{Claim}[theorem]

\theoremstyle{definition}
\newtheorem{littheorem}{Literature Theorem}[section]
\newtheorem{definition}{Definition}[section]

\usepackage[titlenumbered,vlined,ruled]{algorithm2e}
\DeclareMathOperator*{\argmin}{arg\,min}

\title{Non-convex regularization based on shrinkage penalty function}

\author{Manu Ghulyani \& Muthuvel Arigovindan}
\begin{document}
\maketitle

{Department of Electrical Engg., Indian Institute of Science, Bengaluru-12, Karnataka, India.}\\
{manug@iisc.ac.in \& mvel@iisc.ac.in }
\vspace{10pt}

\begin{abstract}
	Total Variation regularization (TV) is a seminal approach for image recovery. TV involves the norm of the image's gradient, aggregated over all pixel locations. Therefore, TV leads to piece-wise constant solutions, resulting in what is known as the "staircase effect." To mitigate this effect, the Hessian Schatten norm regularization (HSN) employs second-order derivatives, represented by the pth norm of eigenvalues in the image hessian vector, summed across all pixels. HSN demonstrates superior structure-preserving properties compared to TV. However, HSN solutions tend to be overly smoothed. To address this, we introduce a non-convex shrinkage penalty applied to the Hessian's eigenvalues, deviating from the convex lp norm. It is important to note that the shrinkage penalty is not defined directly in closed form, but specified indirectly through its proximal operation. This makes constructing a provably convergent algorithm difficult as the singular values are also defined through a non-linear operation. However, we were able to derive a provably convergent algorithm using proximal operations. We prove the convergence by  establishing that the proposed regularization adheres to restricted proximal regularity. The images recovered by this regularization were sharper than the convex counterparts.
	\end{abstract}

%
%
%
%
%

\section{Introduction}
Total Variation (TV) \cite{rudintv} is widely applicable because of its ability to preserve edges. But, image reconstruction (restoration) via TV regularization leads to piece-wise constant estimates. This effect is known as the staircase effect. A well-known workaround for the above problem is to use higher-order derivatives \cite{Scherzer_tv2_98, tgv, lysaker2003noise, hessian} of the image rather than only the first-order derivative. The use of higher-order derivatives leads to smooth intensity variations on the edges rather than sharp jumps in the intensity, thereby eliminating staircase artifacts. The above workaround led to many solutions, important ones being TV-2, Total Generalized Variation (TGV), Hessian-Schatten norm, and others. Hessian-Schatten (HS) norm regularization is an important work because of it theoretical properties and good performance for a wide variety of inverse problems \cite{hsn_poisson,cbct,hsn_eg}. Although HS norm regularization leads to good reconstruction quality, it leads to smoothing of solution-images, which is a common drawback of all convex regularizations. Also, it is well known that non-convex regularizations \cite{pnas} lead to sharper images. But, convergence of non-convex and non-smooth optimization algorithms is difficult to establish. Due to these difficulties in optimization of non-convex functionals, there are very few works that explore higher-order derivative based non-convex regularization functionals, and also have convergence guarantees. An important work by \cite{nikolova_tip} explores the properties non-convex potential functionals, and also give an algorithm to solve the image reconstruction problem. It is also important to mention the work by \cite{nikolova2005analysis} that analyses properties of edges of the recovered images via non-convex regularization functionals. There are many works that explore non-convex first-order total variation, for e.g. \cite{chartrand_tpv,shrinkage}. The work by \cite{shrinkage} is important as it provides convergence as well as recovery guarantees of the proposed reconstruction algorithm. To the best of our knowledge, there is no non-convex regularization that exploits the structural information encoded in the singular values of the image hessian. This is because computing these singular values involves a non-linear operation without a known closed-form solution. In this work, we derive a non-convex regularization inspired from the Hessian-Schatten norm \cite{hessian} and non-convex shrinkage penalty \cite{shrinkage}. In this work, we use the shrinkage penalty on the singular values of the hessian. Although non-convex regularizations are designed to better approximate the $l_0$ norm, non-convexity has many drawbacks such as convergence issues and no recovery guarantees. In addition to being non-convex, the optimization problem for the non-convex formulation of  HS is also non-smooth. This further adds to the complexity of the problem in terms of optimization. Now, the above problems are solved by the following contributions in this work:
\begin{enumerate}
	\item Design of a non-convex regularization retaining  the theoretical and structural properties of the original HS-norm,
	\item algorithm to solve the image restoration problem with the proposed non-convex functional,
	\item convergence results for the algorithm, and 
	\item establshing various theoretical properties of the restoration cost with the proposed regularization.
	\subsection{Organization of the paper} In \cref{ch2:sec:formulation}, we give describe the proposed non-convex regularization. In \cref{ch2:sec:restoration}, we give details of the image restoration problem and show the numerical results for the image restoration problem in \cref{ch2:sec:simulations}. Finally, all the theoretical results and proofs are given in \cref{ch2:sec:theory}.
\end{enumerate} 
\section{Formulation}\label{ch2:sec:formulation}
\subsection{Forward model}The degradation model for a linear imaging inverse problem is expressed as follows:

\begin{equation}
\mathbf{m} = \mathcal{T}(\mathbf{u}) + \mathbf{\eta},
\end{equation}

In this work, our approach involves considering images in a (lexicographically) scanned form, departing from the conventional 2-D array perspective. Therefore, the measurement image takes on a vector representation: $\mathbf{m} \in \mathbb{C}^{N}$, while $\mathbf{u} \in \mathbb{R}^{N}$ signifies the  original image specimen. Additionally, the operator $\mathcal{T}$ is a linear operator representing the forward model.

This paper focuses on MRI image reconstruction. In MRI reconstruction, the forward model $\mathcal{T}$ can be understood as the composition of two operators: $\mathcal{T} = \mathcal{M} \circ \mathcal{F}$. The operator $\mathcal{M}$ corresponds to the sampling trajectory and can be represented through multiplication by a diagonal matrix, which embodies a 2D mask consisting of 1s (where sampling occurs) and zeros (where sampling is absent). On the other hand, $\mathcal{F}$ symbolizes the 2D Discrete Fourier Transform. Additionally, $\mathbf{\eta}\in \mathbb{C}^{N}$ denotes the Gaussian measurement noise.

It is important to note that when referring to a pixel in the square image $\mathbf{u}$, containing $N$ pixels, at the coordinates $[r_1,r_2]$, the notation $[\mathbf{u}]_{\mathbf{r}}$ is used, rather than $u_{r_1\sqrt{N}+r_2}$. This choice of notation is made to enhance clarity and conciseness when indicating access to the pixel positioned at coordinate $\mathbf{r}$. Also, $\sum_{\bn r}$ denotes the summation over all pixel locations.

\subsection{Hessian-Schatten norm regularization}
\label{sec:format}

The $q$-Hessian-Schatten-norm \cite{hessian} ($\mathcal{HS}_q(\cdot)$) at any pixel location of an image is defined as the $l_q$ norm of the singular values of the image Hessian. The corresponding $q$-Hessian-Schatten-norm (HS) regularization functional for an image is obtained as the sum of these norm values across all pixel locations. Let $\bn D_{xx},\ \bn D_{xy},\ \bn D_{yx} \,\text{and }\bn D_{yy.}$ denote the discrete second derivative operators (i.e, discrete analogue of second-order partial derivative $\partial(\cdot)/\partial x\partial y$ etc.), then the discrete Hessian , $\mathcal H:\mathbb R^N\rightarrow \mathbb R^{ 2 \times 2 \times N}$ can be defined as:
$$[\mathcal H(\mathbf u)]_{\mathbf r}=\begin{pmatrix}
[\bn D_{xx}(\mathbf u)]_{\mathbf r} & [\bn D_{xy}(\mathbf u)]_{\mathbf r} \\
[\bn D_{yx}(\mathbf u)]_{\mathbf r} & [\bn D_{yy}(\mathbf u)]_{\mathbf r} 
\end{pmatrix}\in \mathbb R^{2\times 2},$$
for all $N$ pixel locations indexed by $\r$. Now, let $\sigma_1([\mathcal H(\mathbf u)]_{\mathbf r})$ and $\sigma_2([\mathcal H(\mathbf u)]_{\mathbf r})$ denote the singular values of $[\mathcal H(\mathbf u)]_{\mathbf r}$. With this, the Hessian-Schatten-norm can be defined as:

$$\mathcal{HS}_q(\mathbf u)=\sum_{\bn r}\big[|\sigma_1([\mathcal H(\mathbf u)]_{\mathbf r})|^q+ |\sigma_2([\mathcal H(\mathbf u)]_{\mathbf r})|^q\big]^{1/q},$$ 
where $ q$ is considered to lie in $[1,\infty].$ This is because the above choice of $q$ makes $\mathcal{HS}_q(\cdot)$ convex, and therefore efficient convex optimization  algorithms (e.g. ADMM\cite{hsn_poisson}, primal-dual splitting \cite{hessian} etc.) can be used to obtain the reconstruction. The original work \cite{hessian} also proposed solutions using proximal operators for $q\in \{1,2,\infty\}$. Although the convexity of the HS-norm described above is an advantage with respect to optimization and convergence, it has been verified theoretically as well as by numerical experiments that non-convex regularization functionals lead to a better quality of the recovered image.   This motivates the extension of the HS penalty to non-convex formulation, which can lead to better reconstruction.  In this following section, we describe the  non-convex  HS (based) regularization obtained by applying the shrinkage penalty \cite{shrinkage} on the singular values of the Hessian. 
\subsection{Shrinkage Penalty}\label{ch2:sec:qshs}
The shrinkage penalty, as discussed in \cite{shrinkage}, possesses theoretical properties—such as the exact recovery of sparse vectors and the convergence of proximal algorithms—that resemble those of the conventional \(l_1\) penalty, despite its non-convex nature.

The shrinkage penalty (\(g_q(\cdot)\)) is not explicitly defined; rather, it is characterized by its proximal operation. The proximal operation (\(s_q\)) is the solution to the following cost for \(t\):
\begin{align}\label{sq}
\gamma(x,t) =  \rho g_q(t) + \frac{1}{2}(t-x)^2.
\end{align}
The function \(s_q\)  is given by:
\[
s_q(x) = \argmin_t \gamma(x,t) = \max\left\{|x| - \rho^{2-q}|x|^{1-q}, 0\right\} \cdot \text{sign}(x).
\]
This expression of \(s_q(\cdot)\) is known as the \(q\)-shrinkage operation. Notably, when \(q=1\), \(s_q(\cdot)\) corresponds to the familiar soft-thresholding operation (which is the solution to \cref{sq} with \(g_q(t)\) replaced by \(|t|\)). For a given proximal mapping ($s$), the corresponding cost ($g$) will exist given the conditions  outlined in \cref{lit1} are satisfied. 

\begin{littheorem}\cite{shrinkage} \label{lit1}Consider a continuous function $s:[0,\infty)\rightarrow \mathbb R$ and satisfies $$s(x)=\begin{cases}0 \ \ x\le \lambda  \\
	\text{strictly increasing} \ \ \ x\ge\lambda,
	
\end{cases}$$ also $s(x)\le x.$ With this $s$, define $S:\mathbb R \rightarrow \mathbb R$ such that  $S: x \mapsto s(x)sign(x)$, then $S(\cdot)$ is a proximal mapping of an even, continuous and strictly increasing function $g.$ Moreover, $g(\cdot)$ is differentiable in $(0,\infty)$, and $g(\cdot) $ is non-differrentiable at $0$ if and only if $\lambda >0$ with $\partial g(0)=[-1,1]$.

\end{littheorem}

It can be observed that these conditions are satisfied by the \(s_q\) defined above.
The function ($g_q(\cdot)$) derived from the shrinkage function    has some interesting properties:
\begin{itemize}
\item $g_q(\cdot)$ is coercive for $q\in (0,1)$

\item $g_q''(x)<0$ for all $x\in (0,\infty)$, this means that $g_q':(0,\infty)\rightarrow (0,\infty)$ is invertible and $(g_q')^{-1}:(0,\infty)\rightarrow (0,\infty)$ is well defined. 
\end{itemize} 
These properties will be used in showing the existence of the solution for the regularized image reconstruction, and the restricted proximal regularity of the cost.

\subsection{QSHS: q-Shrinkage Hessian-Schatten penalty }
With this $g_q$, we can define the shrinkage-Schatten penalty ($f(\cdot)$) on the singular values of the image Hessian $\mathcal{H}(\mathbf{u})$ at pixel location $\mathbf r$ as:
\begin{align}\label{shrink_singular}f([\mathcal H\mathbf u]_{\r})=g_q(\sigma_1([\mathcal H\mathbf u]_{\mathbf r}))+g_q(\sigma_2([\mathcal H\mathbf u]_{\mathbf r})).\end{align}
Without the closed form solution for $g_q$, we can still use $g_q$ for image restoration as we can solve the following optimization problem (which is a step  in the ADMM algorithm described in \cref{step}) in terms of the shrinkage operation  $ s_q(\cdot).$

$$ H^*=\argmin_{  H} \frac{1}{2}\| M- H\|_2^2+\rho f( H),  M\in \mathbb R^{2\times 2} .$$ 
To solve the above problem, we define $  M=  U  S  V^T$ and $  H=  U_1  S_1  V_1^T $ to be the singular value decompositions of $  M$ and $  H$ respectively. Following the approach by \cite{hessian}, we apply Von Neumann's trace inequality to obtain: $\|  M-  H\|_2\ge \|  S-  S_1\|_2$. Now, using the result obtained above, we get:
$\rho f(  H)+\frac{1}{2}\|  M-  H\|_2^2\ge  0.5(\sigma_2(  H)-\sigma_2(  M))^2+0.5(\sigma_1(  H)-\sigma_1(  M))^2+\rho g_q(\sigma_1(  H))+ \rho g_q(\sigma_2(  H))$. Since, the problem is separable we can obtain $$\sigma_i^*=\argmin_{\sigma}\frac{1}{2}(\sigma -\sigma_i(  M))^2+\rho g_q(\sigma)=s_q(\sigma_i(  M)),$$ for $i=1,2$. Therefore,   
\begin{align}\label{shrink}
{H^{*}}=  U \begin{pmatrix}
s_q(\sigma_1(  M)) & 0\\
0& s_q(\sigma_2(  M)) \end{pmatrix}   V^T.
\end{align}
Here, the proposed $ {H^{*}}$ will have a sparser set of singular values, when compared to the original $ {H}$. By this formulation it is clear that  the proposed non-convex functional will retain the properties of the original HS formulation and lead to sharper results.

The above defined $f(\cdot)$ has many propoerties similar to HS norm:
\begin{itemize}
\item We prove that the QSHS penalty satisfies a technical condition of the so-called restricted proximal regularity (please refer to \cref{def:resp} for details), which generalizes the concept of convexity.  This condition helps us to show the convergence of the proposed algorithm. We prove this result in  \cref{ch2:sec:theory}.
\item The (continuous analogue of) QSHS penalty is translational and rotational invariant. We state the result rigorously in the form of the following proposition. 
\begin{restatable}{proposition}{rotationinvariant}\label{ch2:thm:rotation-invariant}
	Let $u:\mathbb R^2\rightarrow \mathbb R$ be a twice continuously differentiable function and let $\mathcal H$ denote the hessian operator, $f(u) \df \int_{\mathbf r}g_q(\sigma_1(\mathcal Hu(\mathbf r)))+g_q(\sigma_2(\mathcal Hu(\mathbf r)))d\mathbf r$, then $f(u)=f(u\circ R_{\theta})$ for any rotation matrix $R_{\theta}.$ 
\end{restatable}
The proof of the above proposition is similar to the one presented in \cite{hessian}. The result given in \cite{hessian} can be directly extended to our formulation as the proposed penalty is based on the singular values of the hessian. Therefore, we skip the proof.

\end{itemize}

\section{Image restoration problem}\label{ch2:sec:restoration}
The recovered image ($\mathbf u^*$) can be obtained by solving the following optimization problem:
\begin{align}\label{main_opt}
\mathbf u^{*}=\argmin_{u \in S}\frac{1}{2}\|\mathcal T \mathbf u-\mathbf m\|_2^2+\rho \sum_{\mathbf r}f([\mathcal H\mathbf u]_{\mathbf r}).
\end{align}

Here, $f$ is the shrinkage penalty defined in $\cref{shrink_singular}$ and $S$ is the set where desired solution lies. For example, one widely used choice for $S$ is the positive orthant, i.e $\{\mathbf u| [\mathbf u]_{\mathbf r}\ge 0 \ \ \forall \mathbf r\}.$ We prove the following lemma that guarantees the existence of the solution of the optimization problem given in  \cref{main_opt}.
\begin{restatable}{lemma}{coercive}\label{ch2:lemma:restoration-coercive}
If $\mathcal N(\mathcal T)\cap \mathcal N(\mathcal  H)=\{\mathbf 0\}$,	the image restoration cost $f(\mathbf u)=\frac{1}{2}\|\mathcal T\mathbf u-\mathbf m\|^2+\rho \sum_{\mathbf r}g_q(\sigma_1([\mathcal H \mathbf u]_{\mathbf r}))+g_q(\sigma_2([\mathcal H \mathbf u]_{\mathbf r}))$ is coercive.
\end{restatable}
Also, $f(\cdot)$ is continuous, therefore, existence of the minimum point is guaranteed by the Weierstrass theorem. For the complete proof, please refer to \cref{ch2:sec:theory:restoration}.

We solve the optimization by ADMM approach. Although ADMM is (conventionally) guaranteed to converge for convex functions, but there are recent works (e.g. \cite{wang2019global}) that demonstrate the effectiveness of ADMM  for non-convex problems. 
\label{sec:pagestyle}
In order to derive the ADMM algorithm, we first write a constrained formulation of \cref{main_opt}. It can be verified that \cref{main_opt} is equivalent to the following constrained problem ($I_S$ denotes the indicator function on set $S$):\begin{align} \nonumber \mathbf u^{*}&=\argmin_{u \in S}\frac{1}{2}\|\mathcal T \mathbf u-\mathbf m\|_2^2+\rho\sum_{\mathbf r}f([\mathbf H]_{\mathbf r})+I_S(\mathbf v),\\& \text{ subject to } [\mathcal H\mathbf u]_{\mathbf r}=[\mathbf H]_{\mathbf r} \ \forall \mathbf r \text{ and } \mathbf u=\mathbf v.
\end{align}
Note that the constrained formulation decouples the two terms in \cref{main_opt}.  The ADMM involves minimization of the augmented Lagrangian ($\mathcal L()$) which is given as:
\begin{align}\label{ch2:eq:lagrangian}\mathcal L(&\mathbf u,\mathbf H,\mathbf v,\hat{\mathbf u}, \hat{\mathbf H})=\\&\nonumber 0.5 \ \|\mathcal T \mathbf u-\mathbf m\|_2^2+\rho\sum_{\mathbf r}f([\mathbf H]_{\mathbf r})+I_S(\mathbf v)\\&\nonumber +\frac{\beta}{2}\sum_{\mathbf r}\|[\mathcal H\mathbf u]_{\mathbf r}-[\mathbf H]_{\mathbf r}\|_F^2+\langle[\hat{\mathbf H}]_{\mathbf r},\|[\mathcal H\mathbf u]_{\mathbf r}-[\mathbf H]_{\mathbf r}\rangle\nonumber +\frac{\beta}{2}\|\mathbf u-\mathbf v\|_2^2+\langle\hat{\mathbf u},\mathbf u-\mathbf v\rangle.\end{align}
The ADMM algorithm is composed of minimization of $\mathcal L()$  w.r.t $\mathbf u,\mathbf H$ and $\mathbf v$ cyclically, and then updating Lagrange multipliers $\hat{\mathbf H}$ and  $\hat{\mathbf u}.$

For any iteration $k \in \mathbb{N}$, the algorithm advances through the following four steps:

\textbf{Step 1, minimization w.r.t $\mathbf{v}$:}
$\mathbf v$ is updated as: $ \mathbf v^{(k+1)}=\argmin_{\mathbf v}\mathcal L(\mathbf u^{(k)},\mathbf H^{(k)},\mathbf v,\hat{\mathbf u}^{(k)}, \hat{\mathbf H}^{(k)}).$ This reduces to the following minimization on completing the squares:\begin{align}
\mathbf v^{(k+1)}&\nonumber =\argmin_{\mathbf v}I_S(\mathbf v) +\frac{\beta}{2} \ \|\mathbf u^{(k)}-\mathbf v+\frac{\hat{\mathbf u}^{(k)}}{\beta}\|_2^2\\&=P_S(\mathbf u^{(k)} +\frac{\hat{\mathbf u}^{(k)}}{\beta})
\end{align}
Here, $P_S(\cdot)$ is the projection on set $S$.\\
\textbf{Step 2, minimization w.r.t $\mathbf{H}$:}\label{step}
In step 2, $\mathbf H$ is updated as: $$ \mathbf H^{(k+1)}=\argmin_{\mathbf H\in \mathbb R^{2\times 2 \times N}}\mathcal L(\mathbf u^{(k)},\mathbf H,\mathbf v^{(k+1)},\hat{\mathbf u}^{(k)}, \hat{\mathbf H}^{(k)}).$$
Since, the above minimization is separable for each pixel location $\mathbf r,$ we solve the minimization for a fixed $\mathbf r.$ This minimization reduces to
$$\argmin_{[\mathbf H]_{\mathbf r}\in \mathbb R^{2\times 2}}\frac{\beta}{2}\|[\mathcal H\mathbf u^{(k)}]_{\mathbf r}-[\mathbf H]_{\mathbf r}+\frac{[\hat{\mathbf H}^{(k)}]_{\mathbf r}}{\beta}\|_F^2+\rho f([\mathbf H]_{\mathbf r}).$$
The solution to the above problem has already been done in the previous section (\cref{shrink}), where $ M$ plays the role of $[\mathcal H\mathbf u^{(k)}]_{\mathbf r}+[\hat{\mathbf H}^{(k)}]_{\mathbf r}.$
\\
\textbf{Step 3, minimizing w.r.t $\mathbf u$:} Updating $\mathbf u$ is essentially minimizing $  \frac{1}{2} \ \|\mathcal T \mathbf u-\mathbf m\|_2^2 +\frac{\beta}{2}\sum_{\mathbf r}\|[\mathcal H\mathbf u]_{\mathbf r}-[\mathbf H^{(k+1)}]_{\mathbf r}+\frac{[\hat{\mathbf H}^{(k)}]_{\mathbf r}}{\beta}\|_F^2\nonumber +\frac{\beta}{2}\|\mathbf u-\mathbf v^{(k+1)}+\frac{\hat{\mathbf u}^{(k)}}{\beta}\|_2^2$ w.r.t $\mathbf u.$ The minimizer of the above cost can be written as:
\begin{align}\big[\mathcal T^*\mathcal T+\beta& (\bn D_{xx}^T\bn D_{xx}+\bn D_{xy}^T\bn D_{xy}+\bn D_{yx}^T \bn D_{yx}+\bn D_{yy}^T \bn D_{yy}+\mathcal I)\big]\bn u^{(k+1)}= \\&\mathcal T^*\bn m+\beta \bn v^{(k+1)}-\bn u^{(k)}+\beta(\bn D_{xx}^T\bar{\bn H}_{11}+\bn D_{xy}^T \bar{\bn H}_{12}+\bn D_{yx}^T \bar{\bn H}_{21}+\bn D_{yy}^T\bar{\bn H}_{22} ).\end{align}
Here, $\left [\begin{array}{cc}\bar{\bn H }_{11}&\bar{\bn H }_{12}\\ \bar{\bn H }_{21}&\bar{\bn H }_{22} \end{array}  \right ]=\bn H^{(k+1)}-\hat{\bn H}^{(k)}/\beta.$The above problem is linear, and as all the operators involved are block circulant, the equation can be solved efficiently by using 2D-FFTs.

\textbf{Step 4, updating multipliers:}
After the above steps multipliers can be updated as:

\begin{itemize}
\item $\hat{\mathbf u}^{(k+1)}=\hat{\mathbf u}^{(k)}+\beta(\mathbf u^{(k+1)}-\mathbf v^{(k+1)})$, and
\item $\forall \mathbf r: [\hat{\mathbf H}^{(k+1)}]_{\mathbf r}=[\hat{\mathbf H}^{(k)}]_{\mathbf r}+\beta[\mathcal H\mathbf u^{(k+1)}]_{\mathbf r}-\beta[\mathbf H^{(k+1)}]_{\mathbf r}.$
\end{itemize}
\subsection{Convergence Guarantees}
ADMM was proposed in \cite{glowinski1975approximation,gabay1976dual}. ADMM typically  converges  for convex problems \cite{davis2016convergence}, but can fail to converge for multi-block (3 or more) splitting. The behaviour of ADMM for non-convex and non-smooth problems was largely unknown and many questions are still unanswered. But, owing to successful results of the algorithm in many applications (especially in signal processing literature, see for e.g. \cite{chartrand2013nonconvex,liavas2015parallel,wen2012alternating}) there has been a lot of interest in understanding the convergence of the ADMM for non-convex and non-smooth problems. There are many frameworks that establish the convergence of non-convex ADMM \cite{li2015global,hong2016convergence,wang2019global,magnusson2015convergence,wang2018convergence}. The works by \cite{magnusson2015convergence} and \cite{wang2018convergence}
need restrictive  assumptions on the iterates, which are difficult to verify. The \cite{li2015global} work requires that the hessian of the smooth part (data-fitting) of the cost  be lower-bounded, this is not true as $\mathcal T$ has a non-trivial null space for most of the imaging inverse problems.   \cite{hong2016convergence} prove the convergence for only a special class of optimization problems, and the framework is not general. \cite{wang2018convergence} provide the most general framework and allow us to prove that the algorithm is (subsequentially) convergent.  We prove the following theorem that guarantees that any sub-sequential limit of the sequence generated by the above algorithm is a stationary point of the image restoration cost.
\begin{restatable}{theorem}{convergence}\label{ch2:thm:convergence}If $\mathcal N(\mathcal T)\cap \mathcal N(\mathcal  H)=\{\mathbf 0\}$ and $\beta$ is sufficiently large the iterates generated algorithm defined in \cref{ch2:sec:restoration} by steps 1-4 are bounded. Moreover, each limit point of the sequence generated by the iterate is a stationary point of the image restoration cost $f(\mathbf u)=\frac{1}{2}\|\mathcal T\mathbf u-\mathbf m\|^2+\rho \sum_{\mathbf r}g_q(\sigma_1([\mathcal H \mathbf u]_{\mathbf r}))+g_q(\sigma_2([\mathcal H \mathbf u]_{\mathbf r}))$ (defined in \cref{main_opt}). 
\end{restatable}The theorem provides the following assurance: when the sequence produced by the algorithm converges, its convergence will occur at the point where the sub-gradient of the restoration cost reaches zero. Given that the sequence is bounded, the existence of a converging subsequence is guaranteed. Consequently, the limit of this subsequence will correspond to the point where the sub-gradient of the cost becomes zero.
For the proof of the above theorem, please see \cref{ch2:sec:theory:restoration}.

\section{Simulation Results}\label{ch2:sec:simulations}
To demonstrate the effectiveness of the proposed method, we compare the reconstruction results with q-Hessian Schatten norm \cite{hessian} (for $q$=1 and 2) and TV-1 \cite{rudintv}. Hessian Schatten norm for $q=2$ is popularly known as TV-2. We use two sampling masks ($\mathcal M$) with sampling densities 18 and 9 percent. We add noise with $\sigma=2.5.$ In numerical simulations, we use a data-set with 5 typical MRI-images (see \cref{test}) of size $256\times 256.$ For the proposed shrinkage penalty we use $q=.5$ (in \cref{sq}).  The optimal regularization parameter was tuned (by golden-section method) to obtain minimum Mean Squared Error (MSE). The SSIM scores of the reconstructions are given in the \cref{table}. The table clearly shows that the proposed method performs better than all other methods by a significant margin. To demonstrate the visual difference between the images, we show result of image 2  for mask-2 (shown in \cref{result} and zoomed view in \cref{zoom}). Clearly, the proposed method recovers sharper images as dot like structures are much sharper in the proposed method. Also, from the algorithm it is clear that there is no significant computational cost associated with the $q-$ shrinkage step. 
\begin{table}[]
\begin{tabular}{|l|l|l|l|l|l|}
	\hline
	Im                 & Mask & TV-1  & TV-2  & HS-1  & Proposed \\ \hline
	\multirow{2}{*}{1} & 1    & 0.915 & 0.920 & 0.925 & 0.958    \\ \cline{2-6} 
	& 2    & 0.790 & 0.804 & 0.816 & 0.884    \\ \hline
	\multirow{2}{*}{2} & 1    & 0.960 & 0.964 & 0.966 & 0.980    \\ \cline{2-6} 
	& 2    & 0.893 & 0.903 & 0.905 & 0.953    \\ \hline
	\multirow{2}{*}{3} & 1    & 0.937 & 0.937 & 0.940 & 0.957    \\ \cline{2-6} 
	& 2    & 0.815 & 0.815 & 0.820 & 0.849    \\ \hline
	\multirow{2}{*}{4} & 1    & 0.936 & 0.938 & 0.941 & 0.966    \\ \cline{2-6} 
	& 2    & 0.866 & 0.871 & 0.875 & 0.920    \\ \hline
	\multirow{2}{*}{5} & 1    & 0.924 & 0.930 & 0.932 & 0.949    \\ \cline{2-6} 
	& 2    & 0.816 & 0.817 & 0.825 & 0.859    \\ \hline
\end{tabular}
\caption{Table showing SSIM values of reconstructions}\label{table}
\end{table}\begin{figure}[htb]
\includegraphics[width=8.5cm]{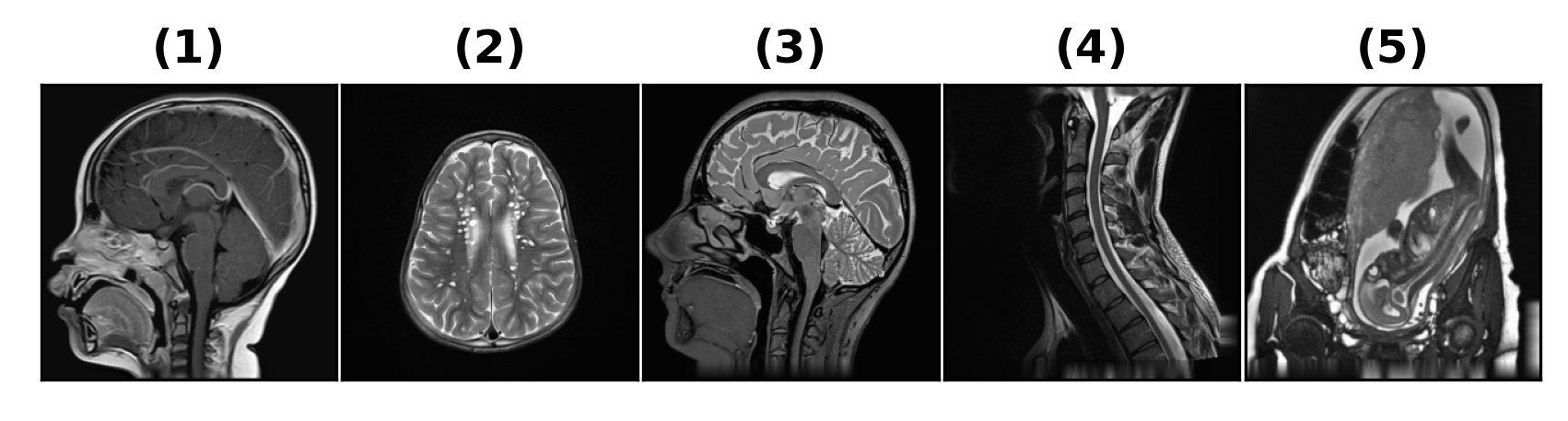}
\caption{Test Images}
\label{test}
\end{figure}
\begin{figure}[htb]
\includegraphics[width=5.5cm]{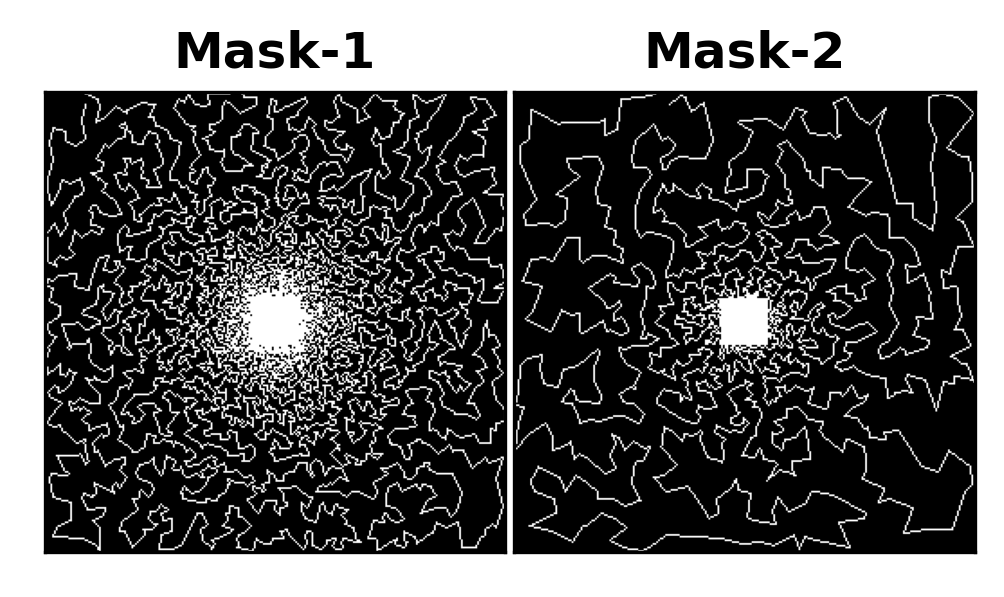}
\caption{Masks for sampling trajectories}
\label{mask}
\end{figure}
\begin{figure}[htb]
\includegraphics[width=8.5cm]{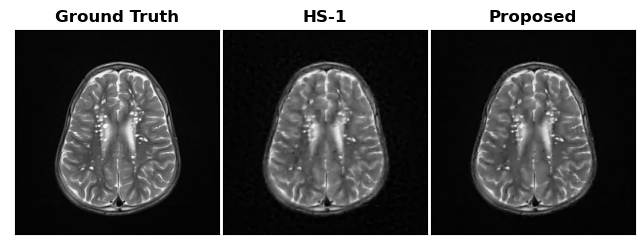}
\caption{Result for Im2 and Mask 2}
\label{result}
\end{figure}
\begin{figure}[htb]
\includegraphics[width=\textwidth]{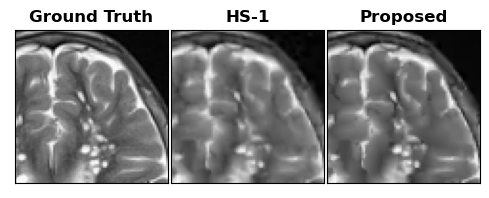}
\caption{Zoomed  view of the result}
\label{zoom}
\end{figure}

\section{Theoretical Results and Proofs}\label{ch2:sec:theory}\subsection{Properties of QSHS penalty
}\label{ch2:sec:theory:qshs}

Now, we formally define the concept of restricted proximal regularity.
\begin{definition}(Restricted proximal regularity) \label{def:resp}A lower semi-continuous function $f:\mathbb R^n\rightarrow \mathbb R \cup \{\infty\}$ is restricted proximal regular if for any $M>0$ and any bounded set $\Omega$ there exists $\gamma\equiv \gamma(M,\Omega)$ such that the following holds for all $y\in \Omega,x\in  \Set{x\in \Omega}{\|p\| \le M \ \ \ \forall p \ \ \ \in \partial f(x) } $, and for all $d\in \partial  f(x)$:
$$f(y)-f(x)-\langle d,(y-x)\rangle\ge-\frac{\gamma}{2}\|y-x\|^2. $$

\end{definition}

%
The following proposition establishes the restricted proximal regularity of the proposed QSHS penalty. The proof proceeds by following the methodology outlined in the proof of restricted proximal regularity for $l_q$ norms ($q \in (0,1)$) as presented in \cite{wang2019global}. However, a notable challenge in this context is the lack of a closed-form expression for the penalty. Consequently, we use the abstract properties of $g_q$ to prove the result.
\begin{restatable}{proposition}{rpr}\label{ch2:thm:rpr}
Consider any $\mathbf H\in \mathbb R^{2\times 2},$ then $r(\bn H)\df g_q(\sigma_1(\mathbf H))+g_q(\sigma_2(\mathbf H))$ is restricted proximal regular.	
\end{restatable}

\begin{proof}We use the following result \cite{watson1992characterization} for the sub-gradient:
Let $H\in \mathbb R^{2\times 2}$ and $H=[U \ U_1] \begin{bmatrix}
S &\mathbf {0}   \\
\mathbf{0} &\mathbf {0}  \\

\end{bmatrix}[V \ V_1]^T$ be the singular value decomposition (SVD) of $H$, $r(H)=\sum_{i=1}^2g_q(S_{ii}(\mathbf H)),$ then $UDV^T+U_1\Theta V_1^T\in \partial r(H),$ where   $D$ is diagonal matrix with entries $(D)_{ii}=g_q'(S_{ii}), \text{ and }\Theta$ is any arbitrary matrix.

Without loss of generality we choose, $\Omega=\Set{X\in \mathbb R^{2\times 2}}{\|X\|\le P}$. Now, for any $M,P>0$, we intend to show that$$r(B)-r(A)-\langle T,B-A\rangle\ge-\frac{\gamma}{2}\|B-A\|^2 $$
for all $B\in \Omega,A \in  \Omega_M \df \Set{X\in \Omega} {\|T\|\le M, \ \ \ \forall T \ \ \ \in \partial r(X)}$,  and for all $T\in \partial r(A) $  		. We do this in the following cases: \begin{description}

\item[Case 1:]$\|B-A\|\ge \epsilon_0=\frac{1}{3}(g_q')^{-1}(M) $ 	 Note that the above condition is equivalent to \begin{align}\label{eq:ch2:case1}\frac{-\|B-A\|}{\epsilon_0}\le-1.\end{align}
First, it can be observed that,
\begin{align}
r(B)-r(A)-\langle T,B-A&\rangle\stackrel{a}{\ge} -r(A)+\|T\|\|B-A\|,\\&\stackrel{b}{\ge} -R_{max}-M\|B-A\|. 
\end{align}
Here, $(a)$ is true as $r(\cdot)$ is non-negative and by Cauchy-Scwartz inequality; while $(b)$ is true as $T$ is bounded and $r(\cdot)$ is a continuous function on a bounded and closed set, therefore, it  attains maximum $R_{max}.$ Now, using \cref{eq:ch2:case1} we obtain
\begin{align}\label{eq:ch2:gamma1}
r(B)-r(A)-\langle T,B-A\rangle{\ge} -\big(\frac{R_{max}+M\epsilon_0}{\epsilon_0^2}\big)\|B-A\|^2.
\end{align}
\item[Case 2:] $\|B-A\|< \epsilon_0$. To prove this, we first define $\Omega'=\Set{T\in \mathbb R^{n\times n}}{\|T\|\le P,\ \min_{i}\sigma_i(T)\ge \epsilon_0}$. \end{description} Now, we decompose $B=U_B\Sigma_BV_B^T$, $B=U_B\Sigma_B^{(1)}V_B^T+U_B\Sigma_B^{(2)}V_B^T=B_1+B_2,$ where all singular values of $ B_1$ are greater than $\epsilon_0.$ Clearly, $B_1\in \Omega'.$ We show that $A\in \Omega'.$ This is proved by contradiction. Assume the contrary that $A\notin \Omega'.$ This means, $\exists \ i$ such that $\sigma_i(A)<\epsilon_0\implies \sigma_i(A)<3\epsilon_0. $ Now, since $g_q'(\cdot)$ is non-increasing, we have $ g_q'(\sigma_i(A))>g_q'(3\epsilon_0)=g_q'((g_q')^{-1}(M))=M. $ Let $A=U_AS_AV_A^T$ be the singular value decomposition of A. Define $T_1\df U_AS'V_A^T,$ where $\{S'\}_{kk}=g_q'(\{S_A\}_{kk})$ for all $k$. Now, by lemma $T_1\in \partial r (A).$ But, $\|T_1\|\ge \|S'\|\ge g_q'(\sigma_i(A))>M. $ This contradicts the fact that $A\in \Omega_M.$ Hence, $A\in \Omega'.$ Now, we define a function, $F:\mathbb R^{n\times n}\rightarrow \mathbb R^{n \times n}$, which is defined as $$F: X \mapsto U_X D_X' V_X^T.$$ Here, $X=U_XD_XV_X^T$ is the singular value decomposition of $X$ and $D_X'$ is a diagonal matrix which is defined as $(D_X')_{ii}=g_q'((D_X)_{ii})$.  Since $F$ is continuous on compact set on $\Omega'$, it is Lipschitz continuous on $\Omega'$ \cite{ding2018spectral}, this means $\|F(B)-F(A)\|\le L\|B-A\|$.  Now, by Taylor's expansion we have:\begin{align}\label{ch2:eq:pf1}
r(B_1)-r(A)-\langle B_1-A,U_AS'V_A^T\rangle\ge \frac{-L}{2}\|B_1-A\|^2 \end{align} . Now, $\|U_2^TU_B\|\le\frac{\|A-B_1\|}{\epsilon_0}$ and $\|V_2^TV_B\|\le\frac{\|A-B_1\|}{\epsilon_0}$ \cite{li2000new}. 
Now, \begin{align}\label{ch2:eq:pf2}\langle U_2^T\Theta V_2,B_1-A\rangle=\langle \Theta ,U_2U_B^T\Sigma_B^{(1)}V_BV_2^T\rangle\ge- \frac{M^2\| B_1-A|\|^2}{\epsilon_0^2}.\end{align}
Also, $r(B_2)- <T,B_2> \ge 0$ and by triangle inequality we get:$\|B_1-A\|\le \|B_1-B\|+\|B-A\|\le 2\|B-A\|.$ Adding \cref{ch2:eq:pf1} and \cref{ch2:eq:pf2} we get,
\begin{align}
r(B)-r(A)-\langle B-A,T \rangle \ge- (\frac{L}{2}+\frac{4M^2}{\epsilon_0^2})\|B-A\|^2. 
\end{align}
\end{proof}

\subsection{Properties of Restoration cost}\label{ch2:sec:theory:restoration}
The following helps us to establish that the restoration cost is coercive.
\begin{claim}\label{ch2:claim:lb}
If $\mathcal N(\mathcal T)\cap \mathcal N(\mathcal  H)=\{\mathbf 0\}$ the function $\|\mathcal T \mathbf u\|+\|\mathcal {H}\mathbf u\|_*\ge \gamma \|\mathbf u\|, \text {where } \gamma>0.$ Here, $\|\mathcal {H}\mathbf u\|_*$ is $\mathcal HS_1(\bn u),$ the conventional $l-1$ Hessian Schatten norm.
\end{claim}

\begin{proof}
The above statement is trivial if $\mathbf u=\mathbf 0.$ If $\mathbf u\neq \mathbf 0$, let $\hat{\mathbf u }=\frac{\mathbf u}{\|\mathbf u\|}$, then $\|\mathcal T \hat {\mathbf u}\|+\|\mathcal H\hat{\mathbf {u}}\|_*\ge \inf_{\|p\|=1}\|\mathcal T  {\mathbf p}\|+\|\mathcal H{\mathbf {p}}\|_*$. Since, $\|\mathbf p\|=1$ is a compact set, $\exists \mathbf p_{min}$ (with $\|\mathbf p_{min}\|=1$) such that $\inf_{\|p\|=1} \|\mathcal T  {\mathbf p}\|+\|\mathcal H{\mathbf {p}}\|_*=\|\mathcal T  {\mathbf p_{min}}\|+\|\mathcal H{\mathbf p_{min}}\|_*$. Define, $\gamma=\|\mathcal T  {\mathbf p_{min}}\|+\|\mathcal H{\mathbf p_{min}}\|_*.$ Clearly, $\gamma\neq 0$ as we will get a vector in intersection of the null spaces, i.e. $\mathbf p_{min}\in \mathcal N(\mathcal T)\cap \mathcal N(\mathcal  H)$, this contradicts the hypothesis. Re-substituting $\hat{\mathbf u }=\frac{\mathbf u}{\|\mathbf u\|}$ completes the proof.
\end{proof}

\coercive*
\begin{proof}
Without loss of generality, we prove the theorem for $\rho=1.$
Consider the level set $\mathcal L_{\eta} (f)\df \Set{\mathbf x}{f(\mathbf x)\le \eta}. $ Now, if $f(\mathbf x)\le \eta \implies \frac{1}{2}\|\mathcal T\mathbf x-\mathbf m\|^2\le \eta.$ Now, by triangle inequality
\begin{align}\label{ch2:eq:corcv1} 
\|\mathcal T\mathbf x\|\le \sqrt{2\eta}+\|\mathbf m\|.
\end{align}
Since, $\mathbf u \in  \mathcal L_{\eta} (f)\implies g_q(\sigma_i([\mathcal H \mathbf u]_{\mathbf r})) \le \eta \ \ \  \forall \mathbf r \text { and } i=1,2.$ As $g_q(\cdot)$ is coercive for $q\in (0,1)$, we have $\sigma_i([\mathcal H \mathbf u]_{\mathbf r})\le M$ for some $M>0.$ This above statement is true because of the fact that any level set of a coercive function is compact. By Taylor's series of $g_q$ around $0$ we can see that: $g_q(\sigma_i([\mathcal H \mathbf u]_{\mathbf r}))=g_q'(\gamma_{i,\mathbf r})\sigma_i([\mathcal H \mathbf u]_{\mathbf r})$ for $\gamma_{i,\mathbf r}\in [0,2M)$. Since, $g_q'(\cdot )$ is decreasing, therefore, $g_q'(\gamma_{i,\mathbf r})\ge g_q'(2M)\df C_{2M}. $ Hence,  $\mathbf x \in  \mathcal L_{\eta} (f) \implies C_{2M}\|\mathcal H\mathbf u\|_* \le \sum_{\mathbf r}g_q(\sigma_1([\mathcal H \mathbf u]_{\mathbf r}))+g_q(\sigma_2([\mathcal H \mathbf u]_{\mathbf r})\le \eta.$ Now, we use the following  \cref{ch2:claim:lb} to show that the level set $\mathcal L_{\eta}(f)$ is bounded. Using, \cref{ch2:claim:lb} we get  \begin{align}  \gamma \|\mathbf u\|\le	\|\mathcal T\mathbf u\|+\|\mathcal H\mathbf u\|_* \le \frac{\eta}{C_{2M}}+\sqrt{2\eta}+\|\mathbf m\|.	\end{align}

This means the level set $\mathcal L_{\eta}(f)$ is bounded. Combining with the fact the level set is closed as $f(\cdot)$ is continuous implies that the level set is compact. As this is true for any level set, therefore $f$ is coercive.\end{proof}
\subsection {Proof of convergence}

\begin{algorithm}\caption{Non-convex ADMM}
\label{ch2:alg:ncadmm}
$N_{iter}\gets 1000$\;
$\bn u^{(0)}=\bn 0$\;
$\bn v^{(0)}=\bn 0$\;
$\bn w^{(0)}=\bn 0$\;

\If{$i\leq N_{iter}$}
{
$\bn u^{(i+1)}\gets\argmin _{\bn u}\mathcal L_{\beta}(\bn u,\bn v^{(i)},\bn w^{(i)})$\;
$\bn v^{(i+1)}\gets\argmin _{\bn v}\mathcal L_{\beta}(\bn u^{(i+1)},\bn v,\bn w^{(i)})$\;
$\bn w^{(i+1)}\gets \bn w^{(i)}+\beta(\bn A\bn u^{(i+1)}+\bn B\bn v^{(i+1)})$\;
$i\gets i+1$\;

}{

}
\end{algorithm}
We  will use the following theorem by \cite{wang2019global} to show the convergencve of our algorithm.
\begin{littheorem}(\cite{wang2019global}, theorem 2.2)\label{ch2:th:wang}
Consider the minimization of the function $\phi(\mathbf u,\bn v)= h(\bn v)+g(\bn v)$ subject to $\bn A \bn u+\bn B \bn v=\bn 0$ by non-convex ADMM algorithm (\Cref{ch2:alg:ncadmm}). Define the augmented Lagrangian, $\mathcal L_{\beta}(\bn u,\bn v, \bn w)\df\phi(\bn u,\bn v)+\bn w^T(\bn A\bn u+\bn B\bn v)+\frac{\beta}{2}\|\bn A\bn u+\bn B\bn v\|^2.$ If \begin{description}
\item[\textbf{C1}]: $\phi(\bn u,\bn v)$ is coercive on the set $\Set{(\bn u,\bn v)}{\bn A \bn u+\bn B \bn v=\bn 0};$
\item [\textbf{C2}]: $Im(A)\subset Im(B)$, where $Im$ denotes the image of the linear operator;
\item [\textbf{C3}]:$\bn A$ and $\bn B$ are full column rank;
\item [\textbf{C4}]: $g$ is restricted proximal regular, and
\item  [\textbf{C5}]:$h$ is Lipschitz smooth,

\end{description}	then algorithm generates a bounded sequence that has atleast one limit point, and each limit point is a stationary point of $\mathcal L_{\beta}(\cdot)$
\end{littheorem}
Now, we show the convergence of algorithm using the above theorem.

\convergence*
\begin{proof}We establish the validity of the aforementioned theorem by fulfilling the conditions (\textbf{(C1)-(C5)}) of \cref{ch2:th:wang}. In comparison to the splitting presented in \cref{ch2:sec:restoration}, we set \(h \equiv \frac{1}{2}\|\mathcal{T}(\cdot) - \mathbf{m}\|_2^2\) and \(g \equiv \rho\sum_{\mathbf r} f(\cdot) + I_S(\cdot)\). The operator \(\mathbf A\) is analogous to a linear operator, satisfying \(\mathcal{A}(\mathbf{u}) - \mathbf{v} = \mathbf{0}\). For each pixel location \(\mathbf{r}\), \(\mathcal{A}:[\mathbf{u}]_{\mathbf{r}} \rightarrow ([\mathcal{H}\mathbf{u}]_{\mathbf{r}}, [\mathbf{u}]_{\mathbf{r}})\). Regarding the constraints, it is evident that \(\mathbf B\) corresponds to the negative identity matrix.

\textbf{(C1)} follows from \cref{ch2:lemma:restoration-coercive}, and \textbf{(C4)} follows from \cref{ch2:thm:rpr}. \textbf{(C2)} is trivial since \(\mathbf B\) is the negative identity. \textbf{(C5)} holds true due to the quadratic nature of the data-fitting cost. To ascertain \textbf{(C3)}, we must demonstrate that \(\mathcal{N}(\mathbf{B}) = \mathbf{0}\) and \(\mathcal{N}(\mathbf{A}) = \mathbf{0}\). Since \(\mathbf{B}\) is the negative identity matrix, \(\mathcal{N}(\mathbf{B}) = \mathbf{0}\). For \(\mathbf{A}\), let \(\mathcal{A}(\mathbf{z}) = \mathbf{0}\), implying \([ \mathcal{H}\mathbf{z} ]_{\mathbf{r}} = \mathbf{0}\) and \([ \mathbf{z} ]_{\mathbf{r}} = \mathbf{0}\) for all \(\mathbf{r}\), which concludes \(\mathbf{z} = \mathbf{0}\). Hence, all the conditions are met. The next step is to establish that the stationary point of the augmented Lagrangian coincides with that of the restoration cost. Suppose \((\mathbf{u}^*, \mathbf{v}^*, \mathbf{w}^*)\) is the stationary point of the augmented Lagrangian; this implies:
\begin{enumerate}
\item $\mathcal A\bn u^*-\bn v^*=\bn 0$,
\item $\nabla h(\bn u^*)+\mathcal A^T\bn w^*=0,$ and
\item $\partial g(\bn v^*)- \bn w^* \ni \bn 0.$
\end{enumerate}
Rearranging 2 and 3 we obtain $\nabla h(\bn u^*)\in -\mathcal A^T\partial g(\bn v^*) $. Now, we use 1 to get \begin{align}&-\nabla h(\bn u^*)-\mathcal A^T\partial (\mathcal A\bn u^*)\ni \bn 0\\&\implies  -\partial f(\bn u^*)\ni \bn 0\\&\implies   \partial f(\bn u^*)\ni \bn 0. \end{align}
\end{proof}
\section{References}
\bibliography{References}{}
\bibliographystyle{iopart-num}

\end{document}